\documentclass[11pt]{article}

\usepackage{amsfonts}
\usepackage{amssymb}
\usepackage{amsmath}
\usepackage{amsthm}
\usepackage{comment}

\newtheorem{theorem}{Theorem}[section]
\newtheorem{lemma}[theorem]{Lemma}
\newtheorem{proposition}[theorem]{Proposition}

\DeclareMathOperator{\Ai}{Ai}
\DeclareMathOperator{\diag}{diag}

\DeclareMathOperator{\im}{Im}
\renewcommand{\Im}{\im}

\theoremstyle{definition}

\newtheorem{remark}[theorem]{Remark}

\numberwithin{equation}{section}

\begin{document}
\pagenumbering{arabic}

\title{The tacnode Riemann-Hilbert problem}
\author{Arno Kuijlaars\footnote{Department of Mathematics, KU Leuven, Celestijnenlaan 200 B bus 2400, 
3001 Leuven, Belgium, e-mail: arno.kuijlaars@wis.kuleuven.be}
}
\date{\ }
\maketitle

\begin{abstract}The tacnode Riemann-Hilbert problem is a $4 \times 4$ matrix valued
RH problem that appears in the description of the local behavior
of two touching groups of non-intersecting Brownian motions. The same
RH problem was also found by Duits and Geudens to describe a new critical 
regime in the two-matrix model.

Delvaux gave integral representations for some of the entries of the
$4 \times 4$ matrix. We complement this work by presenting integral representations
for all of the entries. As a consequence we give an explicit
formula for the Duits-Geudens critical kernel
\end{abstract}

\section{Introduction}
The tacnode Riemann-Hilbert (RH) problem is a $4 \times 4$ matrix-valued RH problem that
first appeared in the asymptotic analysis of two touching groups
of nonintersecting Brownian motions, a so-called tacnode. The positions of
the non-intersecting Brownian motions are a determinantal point process that 
in a double scaling limit around the tacnode leads to the tacnode process.
The one-time correlation functions of the tacnode process were 
expressed in terms of the tacnode RH problem in \cite{DKZ}. The tacnode RH problem
is related to the Hastings-McLeod solution of the Painlev\'e II equation
as was also discussed in \cite{DKZ}.

The tacnode problem was also analyzed in \cite{AFvM,FV,Joh} using different techniques. 
In these papers the tacnode kernel and its multi-time extension are expressed in 
terms of integrals with resolvents of Airy integral operators 
acting on a half-line. These expressions are very different from the RH formulation.

Recently, Delvaux \cite{Del2} made the connection between the two sets of formulas by presenting
integral representations for some of the entries of the solution
of the tacnode RH problem. These entries are exactly the ones that play a role for the
tacnode kernel in \cite{DKZ}. With these explicit formulas Delvaux could make the connection
between the formulas in \cite{DKZ} and the ones by Ferrari and Vet\H{o} \cite{FV} for the asymmetric tacnode.
The paper \cite{Del2} was inspired by the paper \cite{BLS} by Baik, Liechty, and Schehr, where
a connection between different sets of formulas for the maximal height and position 
of the Airy$_2$ process was made.

The aim of this paper is to complement the work of \cite{Del2} by providing integral representations
for all the entries of the tacnode RH problem. Some of these remaining entries appear in
the description of a critical kernel appearing in the two-matrix model 
as shown by Duits and Geudens \cite{DG}. We therefore find explicit integral
formulas for the Duits-Geudens critical kernel. 

In section \ref{sec:statement} we recall the tacnode RH problem with some of its properties,
and in particular the connection with the Hastings-McLeod solution of Painlev\'e II. The
main results of this paper are stated in Theorems \ref{theo:ODEsolutions} and \ref{theo:RHsolution} 
below. We compare the solution of the tacnode RH problem with the explicit solution of
the usual $2\times 2$ matrix-valued RH problem for the Hastings-McLeod solution in section \ref{comparison}.

The proofs  of the results are in section \ref{sec:proofs}. A key role is played by 
Lemma \ref{lem:psisolutions}  that describes solutions to a certain ODE system \eqref{psi:ODE}. 
The proof of this lemma follows along the lines of certain proofs in \cite{Del2}. We give full details about
the calculations in section \ref{sec:appendix}. 
Following \cite{Del2} we briefly mention the tacnode kernel in section \ref{tacnodekernel}.
The implications of Theorem \ref{theo:RHsolution} for the Duits-Geudens critical kernel 
are discussed in section \ref{sec:kernels}.

A variation of the tacnode RH problem for the hard-edge tacnode and 
the chiral two-matrix model appears in \cite{Del1,DGZ}. It may be possible that explicit
integral representations for the solution of these RH problems can be found as well.
Other recent contributions \cite{ACJvM,AJvM,BC,GZ} discuss further connections and properties
of the tacnode process.

\section{Statement of results} \label{sec:statement}

\subsection{The tacnode RH problem}

\begin{figure}[t]
\vspace{14mm}
\begin{center}
   \setlength{\unitlength}{1truemm}
   \begin{picture}(100,70)(-5,2)
       \put(40,40){\line(1,0){40}}
       \put(40,40){\line(-1,0){40}}
       \put(40,40){\line(2,3){15}}
       \put(40,40){\line(2,-3){15}}
       \put(40,40){\line(-2,3){15}}
       \put(40,40){\line(-2,-3){15}}
       \put(40,40){\thicklines\circle*{1}}
       \put(39.3,36){$0$}
       \put(60,40){\thicklines\vector(1,0){.0001}}
       \put(20,40){\thicklines\vector(-1,0){.0001}}
       \put(50,55){\thicklines\vector(2,3){.0001}}
       \put(50,25){\thicklines\vector(2,-3){.0001}}
       \put(30,55){\thicklines\vector(-2,3){.0001}}
       \put(30,25){\thicklines\vector(-2,-3){.0001}}

       \put(60,41){$\Gamma_0$}
       \put(47,57){$\Gamma_1$}
       \put(30,57){$\Gamma_2$}
       \put(16,41){$\Gamma_3$}
       \put(30,20){$\Gamma_4$}
       \put(47,20){$\Gamma_5$}
			
			\put(65,48){$\Omega_0$}
			\put(38,63){$\Omega_1$}
			\put(10,48){$\Omega_2$}
			\put(10,30){$\Omega_3$}
			\put(38,15){$\Omega_4$}
			\put(65,30){$\Omega_5$}
    
       \put(81,40){$\small{\begin{pmatrix}0&0&1&0\\ 0&1&0&0\\ -1&0&0&0\\ 0&0&0&1 \end{pmatrix}}$}
       \put(55,62){$\small{\begin{pmatrix}1&0&0&0\\ -1&1&0&0\\ 1&0&1&1\\ 0&0&0&1 \end{pmatrix}}$}
       \put(-4,62){$\small{\begin{pmatrix}1&1&0&0\\ 0&1&0&0\\ 0&0&1&0\\ 0&-1&-1&1 \end{pmatrix}}$}
       \put(-27,40){$\small{\begin{pmatrix}1&0&0&0\\ 0&0&0&-1\\ 0&0&1&0\\ 0&1&0&0 \end{pmatrix}}$}
       \put(-2,15){$\small{\begin{pmatrix}1&-1&0&0\\ 0&1&0&0\\ 0&0&1&0\\ 0&-1&1&1 \end{pmatrix}}$}
       \put(55,15){$\small{\begin{pmatrix}1&0&0&0\\ 1&1&0&0\\ 1&0&1&-1\\ 0&0&0&1 \end{pmatrix}}$}
  \end{picture}
  \caption{The figure shows the jump contours $\Gamma_k$ in the complex plane and the corresponding jump matrices
   $J_k$, $k=0,\ldots,5$, in the tacnode RH problem. $\Omega_k$ is the sector bounded by the two rays $\Gamma_k$ and $\Gamma_{k+1}$. \label{fig:tacnodejumps}}
\end{center}
\end{figure}
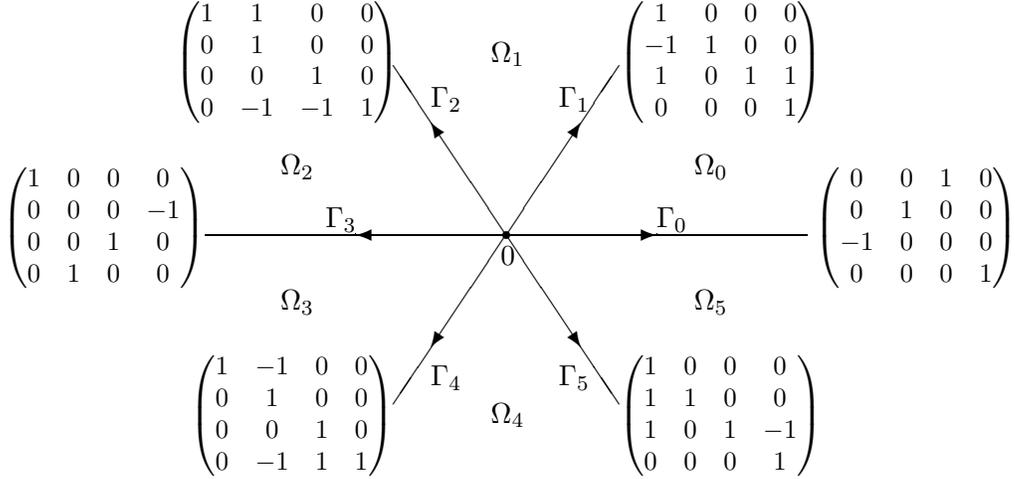

The tacnode RH problem asks for a $4 \times 4$ matrix-valued function
\[ M : \mathbb C \setminus  \Gamma_{M} \to \mathbb C^{4 \times 4} \]
which is defined and analytic outside a set $\Gamma_{M}$ which is a union of six rays
\begin{equation} \label{eq:GammaM} 
	\Gamma_{M} = \bigcup_{k=0}^5 \Gamma_k, \qquad \Gamma_k = \{ z \mid \arg z = \frac{\pi}{3} k \},
	\end{equation}
as shown in Figure \ref{fig:tacnodejumps}. Each ray is oriented from the origin to infinity.
The orientation induces $\pm$-sides on each ray, where the $+$-side is on the left and the $-$-side is on the right
as one traverses the ray according to its orientation.  We ask that $M$ has
continuous boundary values $M_{\pm}$ on each of the rays that satisfy the jump condition
\begin{equation} \label{eq:jumpsM} 
	M_+ = M_- J_k \qquad \text{ on } \Gamma_k \text{ for } k = 0, 1, \ldots, 5,
	\end{equation}
where the jump matrix $J_k$ on $\Gamma_k$ is also shown in Figure \ref{fig:tacnodejumps}. 

The RH problem depends on a number of parameters $ r_1, r_2, s_1, s_2, \tau$ that
appear in the asymptotic condition for $M$ via two functions
\begin{equation} \label{eq:thetas}
\begin{aligned}
\theta_1(z) & = \frac{2}{3} r_1 (-z)^{3/2} + 2 s_1 (-z)^{1/2}, && z \in \mathbb C \setminus [0, \infty), \\
\theta_2(z) & = \frac{2}{3} r_2 z^{3/2} + 2s_2 z^{1/2},  && z \in \mathbb C \setminus (-\infty, 0],
\end{aligned} 
\end{equation}
where we use principal branches for the fractional exponents.
The asymptotic condition is 
\begin{multline}
\label{eq:asymptoticM} M(z) =
\left(I+  \frac{M^{(1)}}{z} + O (z^{-2})\right)
\diag((-z)^{-1/4},z^{-1/4},(-z)^{1/4},z^{1/4})
\\ \times \frac{1}{\sqrt{2}}
\begin{pmatrix} 1 & 0 & -i & 0 \\
0 & 1 & 0 & i \\
-i & 0 & 1 & 0 \\
0 & i & 0 & 1 \\
\end{pmatrix}
\diag\left(e^{-\theta_1(z) + \tau z},e^{-\theta_2(z) - \tau z},e^{\theta_1(z) +\tau z},e^{\theta_2(z) - \tau z}\right).
\end{multline}
The residue matrix $M^{(1)}$ is independent of $z$ but depends on the parameters.

\begin{remark}
In the papers \cite{DKZ,DG} the tacnode RH problem was formulated on a union of ten rays. 
Here we choose to combine the two jumps in each of the open quadrants which reduces the number of 
rays by four. It is easy to see that the two RH problems are equivalent. 
\end{remark}

\begin{proposition} \label{prop:existence}
Suppose that the parameters $r_1, r_2, s_1, s_2, \tau$ are real with  $r_1 > 0$ and $r_2 > 0$.
Then the RH problem \eqref{eq:jumpsM}--\eqref{eq:asymptoticM} has a unique solution.
\end{proposition}

Proposition \ref{prop:existence} was proved  in the case $\tau = 0$ by Delvaux, Kuijlaars, 
and Zhang \cite{DKZ},
and in the case $r_1 = r_2 = 1$, $s_1 = s_2$ with general $\tau$ by Duits and Geudens \cite{DG}.
The proof in \cite{DG} extends to the general case as noted by Delvaux \cite{Del2}. 

% We give an outline of the arguments in the appendix.

\subsection{The Painlev\'e II connection}
The tacnode RH problem is related to the Hastings-McLeod solution of the Painlev\'e II equation, 
as was noted in the cited papers \cite{Del2,DKZ,DG}. The Painlev\'e II equation is
\begin{equation} \label{PainleveII} 
	q'' = t q + 2q^3 
	\end{equation}
and the Hastings-McLeod solution is the unique solution of \eqref{PainleveII} that satisfies
\[ q(t) = \Ai(t) (1+ o(1)) \qquad \text{ as }  t \to +\infty \]
where $\Ai$ denotes the usual Airy function.
We also need the related function
\begin{equation} \label{eq:u} 
	u = q'^2 - t q^2 - q^4, 
	\end{equation}
which satisfies $u' = - q^2$.

\begin{proposition} \label{prop:solutionU}
The solution $M$ of the tacnode RH problem satisfies  a differential equation
\begin{equation} \label{M:ODE}
	\frac{\partial}{\partial z} M = U M,
	 \end{equation}
with a  matrix $U$ that is explicitly given in terms of the Hastings-McLeod
solution $q$ of Painlev\'e II and $u$ from \eqref{eq:u} as follows: 
\begin{multline} U = \label{eq:U} \\
\left(\begin{smallmatrix} \tau - s_1^2 + \frac{u}{C} & \frac{\sqrt{r_2} q}{\gamma \sqrt{r_1} C} & i r_1 & 0 \\
	- \gamma \frac{\sqrt{r_1} q}{\sqrt{r_2} C} & - \tau + s_2^2 - \frac{u}{C} & 0 & i r_2 \\
	i\left(r_1 z - 2s_1 + \frac{s_1^4}{r_1} - \frac{2s_1^2 u}{r_1 C} + \frac{u^2-q^2}{r_1C^2}\right) & 
	\frac{i}{\gamma} \left(\sqrt{r_1 r_2} C (q'+ uq) - \frac{r_1^2 s_2^2 + r_2^2 s_1^2}{(r_1 r_2)^{3/2}} \frac{q}{C} \right) & 
\tau + s_1^2 - \frac{u}{C} & \frac{\sqrt{r_1} q}{\gamma \sqrt{r_2} C} \\	
	i \gamma \left( \sqrt{r_1 r_2} C (q'+ uq) - \frac{r_1^2 s_2^2 + r_2^2 s_1^2}{(r_1 r_2)^{3/2}} \frac{q}{C} \right)& 
	i\left(-r_2 z - 2s_2 + \frac{s_2^4}{r_2} - \frac{2s_2^2 u}{r_2 C} + \frac{u^2-q^2}{r_2C^2}\right) & 
	- \gamma \frac{\sqrt{r_2}q}{\sqrt{r_1}C} & -\tau - s_2^2 + \frac{u}{C}  
	\end{smallmatrix} \right).
	\end{multline}
Here $r_1, r_2, s_1, s_2, \tau$ are the parameters in the problem,
\begin{align} \label{eq:C}
  C & = (r_1^{-2} + r_2^{-2})^{1/3}, \\ 
	\label{eq:gamma}
	\gamma & = \exp \left( \frac{8}{3} \frac{r_1^2 - r_2^2}{(r_1^2 + r_2^2)^2} \tau^3 - 4 \frac{r_1s_1 - r_2 s_2}{r_1^2 + r_2^2} \tau\right),
	\end{align}
and the Painlev\'e functions $q, q'$, and $u$ that appear in \eqref{eq:U} are evaluated in
\begin{align} 	\label{eq:t}
	t & = \frac{2}{C} \left( \frac{s_1}{r_1} + \frac{s_2}{r_2}  -  \frac{2 \tau^2}{r_1^2 + r_2^2} \right).
	\end{align}		
\end{proposition}
\begin{proof}
See Delvaux \cite{Del2}. Note however that the notation in \cite{Del2} is slightly different
from ours. The constant $\tau$ used in \cite{Del2} is equal to $\frac{2}{r_1^2 + r_2^2} \tau$,  
the constant $D$ in \cite{Del2} is equal to $ \frac{\sqrt{r_1}}{\sqrt{r_2}} \gamma$
and $\sigma$ is used instead of $t$.
\end{proof}

It is known that the Hastings-McLeod solution of Painlev\'e II has no poles on the real line
\cite{HM}. Therefore the linear system \eqref{eq:U} is well-defined for every choice of
real parameters $r_1, r_2, s_1, s_2$, and $\tau$.

There are also linear differential equations 
\begin{equation} \label{M:Laxpair} 
	\frac{\partial}{\partial s_1} M = V_1 M, \qquad \frac{\partial}{\partial s_2} M = V_2 M, \qquad 
	\frac{\partial}{\partial \tau} M = W M, 
	\end{equation}
	with explicitly known matrices $V_1$, $V_2$ and $W$. The Painlev\'e II equation \eqref{PainleveII} 
	arises
	as the compatibility condition for \eqref{M:ODE} 	and \eqref{M:Laxpair}, which can be viewed
	as a Lax pair. Note that the usual Lax pair for Painlev\'e II is of size $2 \times 2$,
	see \cite{FN,FIKN}, and also section \ref{comparison} below.

From \eqref{M:ODE} it follows that each column of $M$ is a solution to the linear system of ODEs
\begin{equation} \label{m:ODE}
	\frac{\partial m}{\partial z}  = U m, \qquad m = \begin{pmatrix} m_1 & m_2 & m_3 & m_4 \end{pmatrix}^T
	 \end{equation}
	with $U$ given by \eqref{eq:U}. 
To specify a solution of \eqref{m:ODE} it is enough to give $m_1$ and $m_2$ since 
\begin{align} \label{eq:m3}
	i r_1 m_3 & = m_1' - (\tau - s_1^2 + \frac{u}{C}) m_1 - \frac{\sqrt{r_2} q}{\gamma \sqrt{r_1} C} m_2, \\
	\label{eq:m4}
	i r_2 m_4 & = m_2' + \gamma \frac{\sqrt{r_1} q}{\sqrt{r_2} C} m_1 - (-\tau +s_2^2 - \frac{u}{C}) m_2,
	\end{align}
which follows from the special structure of $U$ in \eqref{eq:U}.

\subsection{Tracy-Widom functions}

We are going to construct six solutions $m^{(j)}$, $j=0, 1, \ldots, 5$, of \eqref{m:ODE}.
The formulas are based on functions that were first introduced by Tracy and Widom \cite{TW,TW2}.

For $t \in \mathbb R$, we use $K_t$ to denote the integral operator on $[t, \infty)$
with the Airy kernel
\begin{equation} \label{Airykernel} 
	\frac{\Ai(x) \Ai'(y) - \Ai'(x) \Ai(y)}{x-y}. 
	\end{equation}
Thus $K_t : L^2([t,\infty)) \to L^2([t, \infty))$ is defined by
\begin{equation} \label{Airyoperator} 
	(K_t f) (x) = \int_t^{\infty}  \frac{\Ai(x) \Ai'(y) - \Ai'(x) \Ai(y)}{x-y} \, f(y) \, dy. 
	\end{equation}
It is known that $I-K_t$ is invertible on $L^2([t,\infty))$, and we define two functions by
\begin{align} \label{Qt} 
	Q_t &= (I - K_t)^{-1} \Ai, \\
	P_t & = (I-K_t)^{-1} \Ai'. \label{Pt}
	\end{align}
We also put
\begin{equation} \label{Rt}
	R_t(x,y) = \frac{Q_t(x) P_t(y) - P_t(x) Q_t(y)}{x-y}.
	\end{equation}

Both $Q_t$ and $P_t$ are continuous (in fact real analytic) functions on $[t, \infty)$.
It is known that
\begin{equation} \label{QtRttt}  
	Q_t(t) = q(t), \qquad R_t(t,t) = u(t),
	\end{equation}
where $q$ is the Hastings-McLeod solution of \eqref{PainleveII} and $u$ is given by \eqref{eq:u},
 see e.g.\ section 2.3  in  \cite{TW2} and in particular formula (25).

\begin{lemma} \label{lem:IIKS}
Both $Q_t$ and $P_t$ extend to entire functions on the complex plane 
and
\begin{align} \label{Qt:asymptotics} 
	Q_t(x) & = (1+ O(x^{-1/2})) \Ai(x) = O\left(x^{-1/4} e^{- \frac{2}{3} x^{3/2}}\right) , \\
	P_t(x) & =  (1 + O(x^{-1})) \Ai'(x) = O\left(x^{1/4} e^{-\frac{2}{3} x^{3/2}} \right), \label{Pt:asymptotics} 
 \end{align}
as $x \to \infty$, uniformly for $-\pi + \varepsilon < \arg x < \pi - \varepsilon$, for every $\varepsilon  >0$.
\end{lemma} 
\begin{proof}

According to the Its-Izergin-Korepin-Slavnov theory on integrable operators \cite{IIKS} 
(see also \cite{Dei}) we have that 
\begin{equation} \label{eq:IIKS1} 
	\begin{pmatrix} Q_t(x) \\ P_t(x) \end{pmatrix} = Y_+(x) \begin{pmatrix} \Ai(x) \\ \Ai'(x) \end{pmatrix}
	=  Y_-(x) \begin{pmatrix} \Ai(x) \\ \Ai'(x) \end{pmatrix}, 
	\qquad x \in (t, \infty). 
\end{equation}
where $Y$ is the unique solution of the RH problem:
\begin{itemize}
\item $Y : \mathbb C \setminus [t, \infty) \to \mathbb C^{2 \times 2}$ is analytic,
\item $Y$ has continuous boundary values $Y_+$ and $Y_-$ on $(t,\infty)$ that satisfy 
\begin{equation} \label{eq:IIKS2} 
	Y_+(x) = Y_-(x) \left( I - 2 \pi i \begin{pmatrix} \Ai(x) \\ \Ai'(x) \end{pmatrix} 
	\begin{pmatrix} \Ai'(x) & - \Ai(x) \end{pmatrix} \right),
	\end{equation}
	for $x \in (t,\infty)$,
\item $Y(x) = O( \log |x-t|)$ as $x \to t$,
\item $Y(x) = I + O(1/x)$ as $x \to \infty$. 
\end{itemize}

Thus $x \in \mathbb C \mapsto Y(x)  \begin{pmatrix} \Ai(x) \\ \Ai'(x) \end{pmatrix}$ provides 
the extension of 
$\begin{pmatrix} Q_t \\ P_t \end{pmatrix}$ into an entire function on the complex plane. 
The lemma then follows because of the asymptotic behavior of $Y$ and the
well known behavior
\begin{equation} \label{Airyasympt}
\begin{aligned} 
	\Ai(x) & = \frac{1}{2\sqrt{\pi} x^{1/4}} e^{-\tfrac{2}{3} x^{3/2}} ( 1 + O(x^{-3/2})) \\  
	\Ai'(x) & = \frac{-x^{1/4} }{2\sqrt{\pi}} e^{-\tfrac{2}{3} x^{3/2}} ( 1 + O(x^{-3/2})) 
	\end{aligned}
	\end{equation}
	as $x \to \infty$, $|\arg x| < \pi - \varepsilon$ of the Airy function.
\end{proof}

By \eqref{Rt} we then also have that $R_t$ extends to an entire function in the complex plane and
for every fixed $y$,
\begin{align} \label{Rt:asymptotics} 
	R_t(x,y) = O\left(x^{-3/4} e^{-\frac{2}{3} x^{3/2}} \right)  
 \end{align}
as $x \to \infty$ with $-\pi + \varepsilon < \arg x < \pi - \varepsilon$ for some $\varepsilon  >0$.

\subsection{Six solutions of \eqref{m:ODE}}

The functions $Q_t(x)$ and $R_t(x,t)$ from \eqref{Qt} and \eqref{Rt} appear 
explicitly in the integral formulas we have for the solutions of \eqref{m:ODE}.
We use
\[ \omega = e^{2\pi i/3}. \]
We also put 
\begin{equation} \label{eq:lambdamu}
	\lambda = \frac{r_2^2 - r_1^2}{r_1^2 + r_2^2} \tau, \qquad  
	\mu = \frac{2}{r_1^2 + r_2^2} \tau
	\end{equation}
and we recall the definitions of $C$, $\gamma$ and $t$ in \eqref{eq:C}, \eqref{eq:gamma}
and \eqref{eq:t}.

\begin{theorem} \label{theo:ODEsolutions}
There are six solutions $m^{(j)}$, $j =0, \ldots, 5$ of \eqref{m:ODE} whose
first and second components are given below. In all cases the third and 
fourth components are as in \eqref{eq:m3}--\eqref{eq:m4}.

The six solutions are given as follows. 
\begin{description}
\item[Solution $m^{(0)}$]
Let $F_0(z) = \Ai\left(r_2^{2/3} z + \frac{2s_2}{r_2^{1/3}} \right) e^{-r_2^2 \mu z}$.
Then
\begin{align} \label{m(0)}
	\begin{pmatrix}  m_1^{(0)}(z) \\ m_2^{(0)}(z) \end{pmatrix}  
	= \sqrt{2\pi} r_2^{1/6} e^{\lambda z} 
 \begin{pmatrix} - \frac{\sqrt{r_2}}{\gamma \sqrt{r_1}} \int_t^{\infty} F_0(z + C(x-t)) Q_t(x) dx \\
	F_0(z) + \int_t^{\infty} F_0(z + C(x-t))  R_t(x,t) dx \end{pmatrix}.
	\end{align}
\item[Solution $m^{(1)}$]
Let $F_1(z) = \omega \Ai\left(\omega \left(r_1^{2/3} z + \frac{2s_1}{r_1^{1/3}} \right) \right) e^{-r_1^2 \mu z}$.
Then
\begin{align} \label{m(1)}
	\begin{pmatrix}  m_1^{(1)}(z) \\ m_2^{(1)}(z) \end{pmatrix}  
	= -\sqrt{2\pi} r_1^{1/6} e^{\lambda z} 
 \begin{pmatrix} 	F_1(-z) + \int_t^{\infty \omega^2} F_1(-z + C(x-t))  R_t(x,t) dx \\
- \gamma \frac{\sqrt{r_1}}{\sqrt{r_2}} \int_t^{\infty \omega^2} F_1(-z + C(x-t)) Q_t(x) dx  \end{pmatrix}.
	\end{align}
\item[Solution $m^{(2)}$]	
Let $F_2(z) = \omega^2 \Ai\left(\omega^2 \left(r_2^{2/3} z + \frac{2s_2}{r_2^{1/3}}\right) \right) e^{-r_2^2 \mu z}$.
Then 
\begin{align} \label{m(2)} 
	\begin{pmatrix}  m_1^{(2)}(z) \\ m_2^{(2)}(z) \end{pmatrix}  
	 = -\sqrt{2\pi} r_2^{1/6} e^{\lambda z} 
			\begin{pmatrix} - \frac{\sqrt{r_2}}{\gamma \sqrt{r_1}} \int_t^{\infty \omega} F_2(z + C(x-t)) Q_t(x) dx \\
	F_2(z) + \int_t^{\infty\omega} F_2(z + C(x-t))  R_t(x,t) dx \end{pmatrix}.
	\end{align}
\item[Solution $m^{(3)}$]
Let $F_3(z) = \Ai\left(r_1^{2/3} z + \frac{2s_1}{r_1^{1/3}} \right) e^{-r_1^2 \mu z}$. Then
\begin{align} \label{m(3)}
	\begin{pmatrix}  m_1^{(3)}(z) \\ m_2^{(3)}(z) \end{pmatrix}  
	 = \sqrt{2\pi} r_1^{1/6} e^{\lambda z} 
 \begin{pmatrix} 	F_3(-z) + \int_t^{\infty} F_3(-z + C(x-t))  R_t(x,t) dx \\
- \gamma \frac{\sqrt{r_1}}{\sqrt{r_2}} \int_t^{\infty} F_3(-z + C(x-t)) Q_t(x) dx  \end{pmatrix}.
	\end{align}
\item[Solution $m^{(4)}$] 
Let $F_4(z) = \omega \Ai\left(\omega \left(r_2^{2/3} z + \frac{2s_2}{r_2^{1/3}}\right) \right) e^{-r_2^2 \mu z}$.
Then
\begin{align} \label{m(4)}
	\begin{pmatrix}  m_1^{(4)}(z) \\ m_2^{(4)}(z) \end{pmatrix}  
	= \sqrt{2\pi} r_2^{1/6} e^{\lambda z} 
 \begin{pmatrix} - \frac{\sqrt{r_2}}{\gamma \sqrt{r_1}} \int_t^{\infty \omega^2} F_4(z + C(x-t)) Q_t(x) dx \\
	F_4(z) + \int_t^{\infty\omega^2} F_4(z + C(x-t))  R_t(x,t) dx \end{pmatrix}.
	\end{align}
\item[Solution $m^{(5)}$]
Let  $F_5(z) = \omega^2 \Ai\left(\omega^2 \left(r_1^{2/3} z + \frac{2s_1}{r_1^{1/3}} \right) \right) e^{-r_1^2 \mu z}$.
Then
\begin{align} \label{m(5)}
	\begin{pmatrix}  m_1^{(5)}(z) \\ m_2^{(5)}(z) \end{pmatrix}  
	= \sqrt{2\pi} r_1^{1/6} e^{\lambda z} 
 \begin{pmatrix} 	F_5(-z) + \int_t^{\infty \omega} F_5(-z + C(x-t))  R_t(x,t) dx \\
- \gamma \frac{\sqrt{r_1}}{\sqrt{r_2}} \int_t^{\infty \omega} F_5(-z + C(x-t)) Q_t(x) dx  \end{pmatrix}.
	\end{align}
\end{description}
\end{theorem}

The proof of Theorem~\ref{theo:ODEsolutions}  is given in section \ref{proof:ODEsolutions}.

\begin{remark}
The solutions $m^{(0)}$ and $m^{(3)}$ were found by Delvaux \cite{Del2}. These are
the solutions for which the integrals are taken over the real interval $(t, \infty)$. 
The other solutions are new and their identification is the main result of this paper.
\end{remark}

\begin{remark}
The integrals in \eqref{m(2)} and \eqref{m(5)} start at $t \in \mathbb R$ and
end at infinity at asymptotic angle $2\pi/3$. Similary, the integrals in \eqref{m(1)} 
and \eqref{m(4)} end at asymptotic angle $-2 \pi/3$.

By Lemma \ref{lem:IIKS} we have that $Q_t(x)$ and $R_t(x,t)$ are entire functions in $x$.
To see that the integrals in \eqref{m(2)} indeed converge, we note
that both $Q_t(x) = O(e^{- \frac{2}{3} x^{3/2}})$ and $R_t(x,t) = O(e^{- \frac{2}{3} x^{3/2}})$ 
by \eqref{Qt:asymptotics} and \eqref{Rt:asymptotics}, and so
\[ F_2(z + C(x-t)) 
	=  e^{\frac{2}{3} (r_2 C^{3/2} x^{3/2})  + O(x)} \]
	as $x \to \infty$ with $\arg x = 2\pi/3$.
Since by \eqref{eq:C}
\[ r_2 C^{3/2} = r_2 (r_1^{-2} + r_2^{-2})^{1/2} > 1, \] 
the integrands  decay at an exponential rate at infinity,
and the integrals in \eqref{m(2)} converge.

Similarly, the integrals in \eqref{m(1)}, \eqref{m(4)} and \eqref{m(5)} converge.
\end{remark}

\subsection{The solution of the tacnode RH problem}

The vector $m^{(j)}$, $j=0, \ldots, 5$ turns out to be the recessive solution of \eqref{m:ODE}
in the sector
\begin{equation} \label{eq:Sj} 
	S_j = \{ z \in \mathbb C \mid   - \tfrac{\pi}{6} + j \tfrac{\pi}{3} < \arg z < \tfrac{\pi}{6} + j \tfrac{\pi}{3} \}.
\end{equation}
Note that $S_j$ is the sector of angular width $\pi/3$ with $\Gamma_j$ as its bisector, see \eqref{eq:GammaM}.
The constant prefactors in the definitions \eqref{m(0)}--\eqref{m(5)} are chosen such that $m^{(j)}$ appears as one of 
the columns of $M$ in the two sectors $\Omega_{j-1}$ and $\Omega_j$ that intersect $S_j$, where it is
understood that $\Omega_{-1} = \Omega_5$. 
As such the vectors $m^{(j)}$ are the building blocks for the solution
of the tacnode RH problem.

Our main result is the following.

\begin{theorem} \label{theo:RHsolution}
The solution of the tacnode RH problem \eqref{eq:jumpsM}--\eqref{eq:asymptoticM} is given by
\[ M = \begin{pmatrix} m^{(3)} & m^{(0)} & m^{(1)} & m^{(2)} \end{pmatrix} \]
in the sector $\pi/3 < \arg z < 2 \pi/3$ around the positive imaginary axis.
In the other sectors it can be found by following the jumps \eqref{eq:jumpsM} 
in the RH problem, and by using the non-trivial relations
\begin{equation} \label{m:consistent}
 m^{(0)} + m^{(3)} = m^{(1)} - m^{(5)} = m^{(2)} - m^{(4)}
\end{equation}
among the six solutions $m^{(0)}, \ldots, m^{(5)}$ of \eqref{m:ODE}.

Explicit expressions for $M$ in all sectors $\Omega_j$ are given  in Figure~\ref{fig:solutionM}.
\end{theorem}

\begin{figure}[t]
\begin{center}
   \setlength{\unitlength}{1truemm}
   \begin{picture}(100,70)(-5,2)
       \put(40,40){\line(1,0){40}}
       \put(40,40){\line(-1,0){40}}
       \put(40,40){\line(2,3){15}}
       \put(40,40){\line(2,-3){15}}
       \put(40,40){\line(-2,3){15}}
       \put(40,40){\line(-2,-3){15}}
       \put(40,40){\thicklines\circle*{1}}
       \put(39.3,36){$0$}
       \put(60,40){\thicklines\vector(1,0){.0001}}
       \put(20,40){\thicklines\vector(-1,0){.0001}}
       \put(50,55){\thicklines\vector(2,3){.0001}}
       \put(50,25){\thicklines\vector(2,-3){.0001}}
       \put(30,55){\thicklines\vector(-2,3){.0001}}
       \put(30,25){\thicklines\vector(-2,-3){.0001}}

       \put(52,47){$\small{\begin{pmatrix} -m^{(5)} & m^{(0)} & m^{(1)} & m^{(2)} - m^{(1)} \end{pmatrix}}$}
       \put(21,64){$\small{\begin{pmatrix} m^{(3)} & m^{(0)} & m^{(1)} & m^{(2)} \end{pmatrix}}$}
			
       \put(-28,47){$\small{\begin{pmatrix} m^{(3)} & -m^{(4)} & m^{(1)} - m^{(2)} & m^{(2)} \end{pmatrix}}$}
       \put(-27,29){$\small{\begin{pmatrix} m^{(3)} & m^{(2)} & m^{(5)} - m^{(4)}  & m^{(4)} \end{pmatrix}}$}
			
       \put(21,14){$\small{\begin{pmatrix} m^{(3)} & m^{(0)} & m^{(5)} & m^{(4)} \end{pmatrix}}$}
       \put(51,29){$\small{\begin{pmatrix} m^{(1)} & m^{(0)} & m^{(5)} & m^{(4)} - m^{(5)} \end{pmatrix}}$}			  
			
			\put(62,56){$\Omega_0$}
			\put(39,54){$\Omega_1$}
			\put(16,56){$\Omega_2$}
			\put(16,22){$\Omega_3$}
			\put(39,24){$\Omega_4$}
			\put(62,22){$\Omega_5$}
  \end{picture}
	\caption{Structure of the solution of the tacnode RH problem. The column vector 
	$m^{(j)}$ is the recessive solution of \eqref{m:ODE} in
	the sector $S_j$.  \label{fig:solutionM}}
\end{center}
\end{figure}
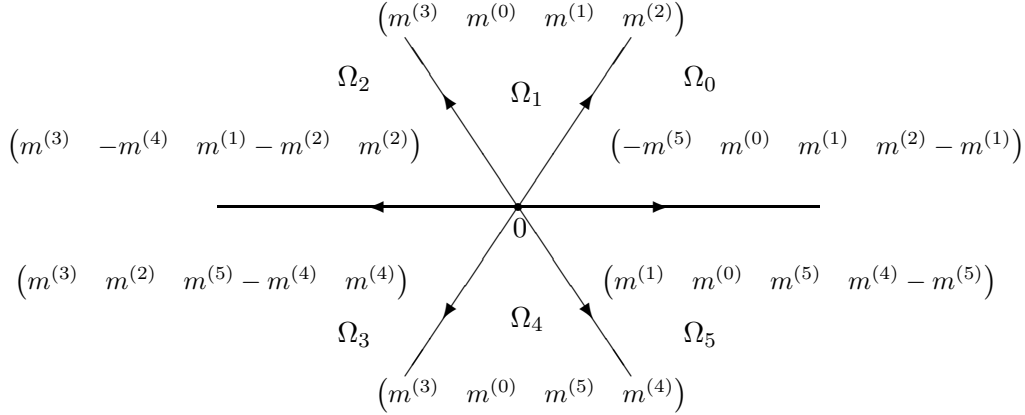
The proof of Theorem \ref{theo:RHsolution} is given in section \ref{proof:RHsolution}.

The relations \eqref{m:consistent} among the recessive solutions of \eqref{m:ODE}
are quite remarkable as they do not follow in a straightforward way from the integral
representations. It is an open problem how to prove these
identities in a direct way from the formulas \eqref{m(0)}--\eqref{m(5)}.

\subsection{Comparison with the $2 \times 2$ RH problem for Painlev\'e II} \label{comparison}

\begin{figure}[t]
\vspace{14mm}
\begin{center}
   \setlength{\unitlength}{1truemm}
   \begin{picture}(100,70)(-5,2)
       \put(40,40){\line(0,1){25}}
       \put(40,40){\line(0,-1){25}}
       \put(40,40){\line(3,2){20}}
       \put(40,40){\line(-3,2){20}}
       \put(40,40){\line(3,-2){20}}
       \put(40,40){\line(-3,-2){20}}
       \put(40,40){\thicklines\circle*{1}}
       \put(38.3,36){$0$}
       \put(40,60){\thicklines\vector(0,1){.0001}}
       \put(40,20){\thicklines\vector(0,-1){.0001}}
    %   \put(60,50){\thicklines\vector(2,1){.0001}}
     %  \put(60,30){\thicklines\vector(2,-1){.0001}}
     %  \put(20,50){\thicklines\vector(-2,1){.0001}}
     %  \put(20,30){\thicklines\vector(-2,-1){.0001}}
       \put(55,50){\thicklines\vector(3,2){.0001}}
       \put(25,50){\thicklines\vector(-3,2){.0001}}
       \put(55,30){\thicklines\vector(3,-2){.0001}}
       \put(25,30){\thicklines\vector(-3,-2){.0001}}

       \put(60,55){$\small{\begin{pmatrix} 1 & 0 \\ s_1 & 1 \end{pmatrix}}$}
       \put(33,70){$\small{\begin{pmatrix} 1 & s_2 \\ 0 & 1 \end{pmatrix}}$}
       \put(6,55){$\small{\begin{pmatrix} 1 & 0 \\ s_3 & 1 \end{pmatrix}}$}
       \put(6,23){$\small{\begin{pmatrix} 1& s_1 \\ 0 & 1 \end{pmatrix}}$}
       \put(33,10){$\small{\begin{pmatrix} 1 & 0 \\ s_2 & 1 \end{pmatrix}}$}
       \put(60,23){$\small{\begin{pmatrix} 1 & s_3 \\ 0 & 1 \end{pmatrix}}$}
  \end{picture}
  \caption{Jump matrices in the $2 \times 2$ RH problem for Painlev\'e II. The Stokes multipliers satisfy 
	$s_1 s_2 s_3 + s_1 + s_2 + s_3 = 0$. The Hastings McLeod solution corresponds to $s_1 = 1$, $s_2 = 0$ and $s_3 = -1$. 
	\label{fig:2x2RHproblem}}
\end{center}
\end{figure}
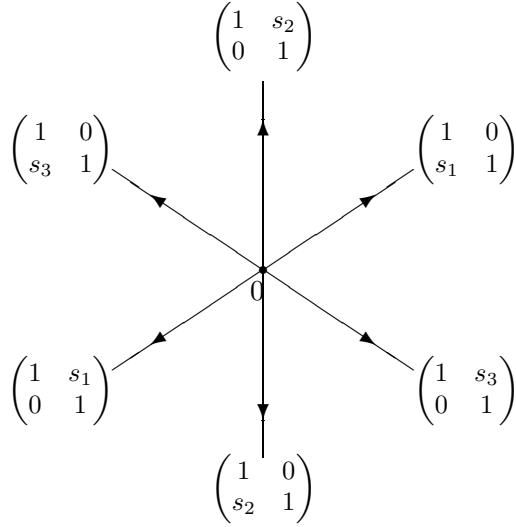

The usual RH problem for Painlev\'e II is of size $2 \times 2$ with a
contour $\Gamma_{\Psi}$ as in Figure \ref{fig:2x2RHproblem}. Then
$\Psi : \mathbb C \setminus \Gamma_{\Psi} \to \mathbb C^{2\times 2}$
is analytic satisfying $\Psi_+ = \Psi_- J_{\Psi}$ on $\Gamma_{\Psi}$ 
with jump matrices $J_{\Psi}$ that are also shown in Figure \ref{fig:2x2RHproblem},
and with the asymptotic condition
\[ \Psi(z) = (I+ O(1/z)) e^{-i (\tfrac{4}{3} z^3 + t z) \sigma_3}, \quad \text{as } z \to \infty \]
where $\sigma_3 = \begin{pmatrix} 1 & 0 \\ 0 & -1 \end{pmatrix}$.
The Stokes multipliers $s_1$, $s_2$, $s_3$ that appear in the jump matrices are
assumed to satisfy $s_1 s_2 s_3 + s_1 + s_2 + s_3 = 0$. Then a unique solution
to the RH problem exists (except for an at most countable number of $t$ values),
and $\Psi$ satisfies the differential equations (Lax pair)
\begin{align} \label{eq:FNLaxpair} 
	\frac{\partial}{\partial z} \Psi & = 
	\begin{pmatrix} - 4i z^2 - i(t + 2q^2) & 4zq + 2i q' \\ 4zq - 2iq' &  4iz^2 + i(t+2q^2) \end{pmatrix} \Psi \\
	\frac{\partial}{\partial t} \Psi & = 
	\begin{pmatrix} - i z & q \\ q  & iz \end{pmatrix} \Psi \nonumber
	\end{align}
	where $q = q(t)$ is a solution of the Painlev\'e II equation, determined by the Stokes
	multipliers, see e.g.\ \cite{FN,FIKN}.

The special case $s_1 = 1$, $s_2 = 0$, $s_3 = -1$ leads to the Hastings McLeod solution
of the Painlev\'e II equation. In this case there is an explicit formula for the solution
of the RH problem, which is contained in the recent work of Baik, Liechty, and Schehr \cite{BLS},
see also \cite{Bai}.
The solution is built out of the two column vectors
\begin{align} \label{psi1} 
	\psi^{(1)}(z) & = e^{-i(\tfrac{4}{3} z^3 + tz)} \begin{pmatrix} 1 + \int_t^{\infty} e^{-2i(x-t) z} R_t(x,t) dx \\
		-\int_t^{\infty} e^{-2i(x-t)z} Q_t(x) dx \end{pmatrix}, \\
	\psi^{(2)}(z) & = e^{i(\tfrac{4}{3} z^3 + tz)} \begin{pmatrix} -\int_t^{\infty} e^{2i(x-t)z} Q_t(x) dx \\
	1 + \int_t^{\infty} e^{2i(x-t) z} R_t(x,t) dx \end{pmatrix},  \label{psi2}
	\end{align}
that satisfy the vector analogue of \eqref{eq:FNLaxpair}.
Then the solution of the RH problem is
\[ \Psi(z) = \begin{pmatrix} \psi^{(1)} & \psi^{(2)} \end{pmatrix} \]
in the two sectors $-\frac{\pi}{6} < \arg z < \frac{\pi}{6}$ and 
$\frac{5\pi}{6} < \arg z < \frac{7\pi}{6}$. The solution in all sectors is
given in Figure~\ref{fig:2x2HastingsMcLeod}.

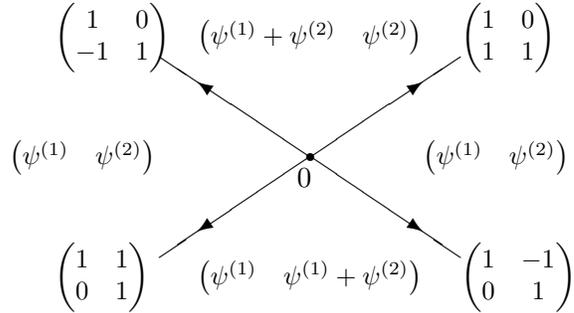
\begin{figure}[t]
\vspace{14mm}
\begin{center}
   \setlength{\unitlength}{1truemm}
   \begin{picture}(100,70)(-5,2)
%       \put(40,40){\line(0,1){25}}
%       \put(40,40){\line(0,-1){25}}
       \put(40,40){\line(3,2){20}}
       \put(40,40){\line(-3,2){20}}
       \put(40,40){\line(3,-2){20}}
       \put(40,40){\line(-3,-2){20}}
       \put(40,40){\thicklines\circle*{1}}
       \put(38.3,36){$0$}
%       \put(40,60){\thicklines\vector(0,1){.0001}}
%      \put(40,20){\thicklines\vector(0,-1){.0001}}
    %   \put(60,50){\thicklines\vector(2,1){.0001}}
     %  \put(60,30){\thicklines\vector(2,-1){.0001}}
     %  \put(20,50){\thicklines\vector(-2,1){.0001}}
     %  \put(20,30){\thicklines\vector(-2,-1){.0001}}
       \put(55,50){\thicklines\vector(3,2){.0001}}
       \put(25,50){\thicklines\vector(-3,2){.0001}}
       \put(55,30){\thicklines\vector(3,-2){.0001}}
       \put(25,30){\thicklines\vector(-3,-2){.0001}}

       \put(60,55){$\small{\begin{pmatrix} 1 & 0 \\ 1 & 1 \end{pmatrix}}$}
 %      \put(33,70){$\small{\begin{pmatrix} 1 & s_2 \\ 0 & 1 \end{pmatrix}}$}
       \put(6,55){$\small{\begin{pmatrix} 1 & 0 \\ -1 & 1 \end{pmatrix}}$}
       \put(6,23){$\small{\begin{pmatrix} 1 & 1 \\ 0 & 1 \end{pmatrix}}$}
  %     \put(33,10){$\small{\begin{pmatrix} 1 & 0 \\ s_2 & 1 \end{pmatrix}}$}
       \put(60,23){$\small{\begin{pmatrix} 1 & -1 \\ 0 & 1 \end{pmatrix}}$}
			
			\put(55,39){$\small \begin{pmatrix} \psi^{(1)} & \psi^{(2)} \end{pmatrix}$}
			\put(0,39){$\small \begin{pmatrix} \psi^{(1)} & \psi^{(2)} \end{pmatrix}$}
			\put(25,55){$\small \begin{pmatrix} \psi^{(1)} + \psi^{(2)} & \psi^{(2)} \end{pmatrix}$}
			\put(25,23){$\small \begin{pmatrix} \psi^{(1)} & \psi^{(1)} + \psi^{(2)} \end{pmatrix}$}
  \end{picture}
  \caption{Solution of the $2 \times 2$ RH problem for the Hastings-McLeod solution of Painlev\'e II
	in terms of the columnn vectors $\psi^{(1)}$ and $\psi^{(2)}$.
	\label{fig:2x2HastingsMcLeod}}
\end{center}
\end{figure}

Comparing this solution with the explicit solution $M$ of the $4 \times 4$ matrix-valued RH problem,
we see that the same functions $Q_t(x)$ and $R_t(x,t)$ appear in the solutions, but there does
not seem to be a direct way to go from $\Psi$ to $M$. The $4 \times 4$ RH problem
and its associated Lax pair therefore provide a genuinely different characterization of the
Hastings-McLeod solution of Painlev\'e II.

\section{Proofs of Theorems \ref{theo:ODEsolutions} and \ref{theo:RHsolution}} \label{sec:proofs}
\subsection{Transformation to second order system}
	
It will be convenient to transform the first order system \eqref{m:ODE} 
to a second order system. The transformation also removes the parameter $\gamma$.
	
\begin{lemma} \label{lem:mandpsi}
Suppose $m = \begin{pmatrix} m_1 & m_2 & m_3 & m_4 \end{pmatrix}^T$ satisfies \eqref{m:ODE}.
Then
\begin{equation} \label{eq:psi}
	\psi(z) = \begin{pmatrix} \psi_1(z) \\ \psi_2(z) \end{pmatrix} 
		= e^{-\lambda z} \begin{pmatrix} \gamma \frac{\sqrt{r_1}}{\sqrt{r_2}} m_1(z) \\ m_2(z) \end{pmatrix}
		\end{equation}
satisfies
\begin{equation} \label{psi:ODE}
 \frac{\partial^2 \psi }{\partial z^2} 
	= A \frac{\partial \psi }{\partial z}  
	+ B \psi
\end{equation}
where 
\begin{align} \label{eq:A}
	A & = \begin{pmatrix} 2 r_1^2 \mu & C^2 r_1^2 q \\  - C^2 r_2^2 q & - 2 r_2^2 \mu \end{pmatrix}  \\
	B & = \begin{pmatrix} -r_1^2 z + 2r_1 s_1 + C r_1^2 q^2 - r_1^4 \mu^2  & 		- C r_1^2 q' \\ 
		- C r_2^2 q' 	& r_2^2 z + 2r_2 s_2 + C r_2^2  q^2 - r_2^4 \mu^2
		 \end{pmatrix}. \label{eq:B}
	\end{align}

Conversely, if $\psi = (\psi_1,\psi_2)^T$ satisfies \eqref{psi:ODE} with $A$ and $B$ given by \eqref{eq:A}-\eqref{eq:B}
then $m = (m_1, m_2, m_3,m_4)^T$ where
\begin{align} m_1(z) & = e^{\lambda z} \gamma^{-1} \frac{\sqrt{r_2}}{\sqrt{r_1}} \psi_1(z), \qquad
	m_2(z)  = e^{\lambda z} \psi_2(z)  \label{m1m2frompsi12}
	\end{align}
and $m_3$, $m_4$ are given by \eqref{eq:m3}--\eqref{eq:m4} satisfies \eqref{m:ODE}.
\end{lemma}
\begin{proof}
This is a straightforward calculation and we will not give full details, see also \cite[Proposition 2.12]{Del2}.

Let us just note that for general $\lambda$ one obtains \eqref{psi:ODE} with
\begin{align*} 
	A = \begin{pmatrix} 2 \tau - 2 \lambda & C^2 r_1^2 q \\  - C^2 r_2^2 q & - 2\tau - 2\lambda \end{pmatrix},  
	\end{align*}
and
\begin{align*}
B = \begin{pmatrix} -r_1^2 z + 2r_1 s_1 + C r_1^2 q^2 - (\tau - \lambda)^2  & 
		- C r_1^2 q' + ((r_1^2 - r_2^2) \tau + (r_1^2 + r_2^2) \lambda)) \frac{q}{r_2^2 C} \\ 
		- C r_2^2 q' - ((r_1^2 - r_2^2) \tau + (r_1^2 + r_2^2) \lambda))	\frac{q}{r_1^2 C}	
		& r_2^2 z + 2r_2 s_2 + C r_2^2  q^2 - (\tau+\lambda)^2  
		\end{pmatrix}
		\end{align*}
Because of the choice \eqref{eq:lambdamu} for $\lambda$
we have that the off-diagonal entries of $B$ simplify
since the terms with $q$ disappear.
By \eqref{eq:lambdamu} it is easy to check that
\[ \tau - \lambda = r_1^2 \mu, \qquad \tau + \lambda = r_2^2 \mu \]
and we obtain \eqref{eq:A} and \eqref{eq:B}.
\end{proof}

\subsection{Solutions to the second order system \eqref{psi:ODE}}

We denote $\omega = e^{2\pi i/3}$ as before, and we let
\[ y_0(x) = \Ai(x), \qquad y_1(x) = \omega \Ai(\omega x), \qquad y_2(x) = \omega^2 \Ai(\omega^2 x) \]
be three solutions of the Airy differential equation $y'' = xy$.

The following lemma gives solutions to the second order system \eqref{psi:ODE}.

\begin{lemma} \label{lem:psisolutions} For $k=0,1,2$, we put 
\begin{equation} \label{Fz} 
	F(z) = y_k\left(r_2^{2/3}z + \frac{2s_2}{r_2^{1/3}}\right) \exp(- r_2^2 \mu z).
	\end{equation}
Then  the vector $\psi =  \begin{pmatrix} \psi_1 \\ \psi_2 \end{pmatrix}$ where
\begin{align} 
  \psi_1(z) & = - \int_t^{\infty \omega^{2k}} F(z + C(x-t)) Q_t(x) dx, \label{eq:psi1} \\
	\psi_2(z) & = F(z)  \label{eq:psi2}
	  + \int_t^{\infty \omega^{2k}} F(z + C(x-t))  R_t(x,t) dx, 
	\end{align}
is a solution of \eqref{psi:ODE}.
\end{lemma}	

\begin{proof} 
For $k=0$, this is proved in \cite[section 5.2]{Del2}, although with somewhat
different notation.
The proof uses the differential equation for $F$ 
\begin{equation} \label{F:ode} 
	F''(z) = -2 r_2^2 \mu F'(z) + (r_2^2 z + 2 r_2 s_2 - r_2^4 \mu^2) F(z),
	\end{equation}
	see also \cite[Equ. (5.2)]{Del2} (where $F$ is called $b_z(x)$ and $\mu$ is called $\tau$),
and the fact that the integrals in \eqref{eq:psi1} and \eqref{eq:psi2} converge for $k=0$. 
We already noted, see Remark 1.6, that the integrals in \eqref{eq:psi1} and \eqref{eq:psi2}
converge for $k=1,2$ as well.

What is also used in \cite{Del2} are a number of identities for the functions 
$Q_t$, $P_t$ and $R_t$ introduced in \eqref{Qt}-\eqref{Rt}, namely 
the differential identities (see also \cite{TW,TW2} or \cite[section 3.8]{AGZ})
\begin{align} \label{Rtdt}
		\frac{\partial}{\partial t} R_t(x,y) & = - R_t(x,t) R_t(t,y), \\
		\label{Qtdt}
		\frac{\partial}{\partial t} Q_t(x) & = - R_t(x,t) Q_t(x), \\
		\label{Ptdt}
		\frac{\partial}{\partial t} P_t(x) & = - R_t(x,t) P_t(x),
		\end{align}
and identities for the $x$-derivatives of $Q_t$ and $P_t$,
\begin{align}
		\label{Qtdx}
		Q_t'(x) & = P_t(x) + q(t) R_t(x,t) - u(t) Q_t(x), \\
		\label{Ptdx}
		P_t'(x) & = x Q_t(x) + p(t) R_t(x,t) + u(t) P_t(x) - 2 v(t) Q_t(x). 
		\end{align}
The identities \eqref{Rtdt}--\eqref{Ptdx}
of course extend into the complex plane.

Also for $k=1,2$, we have that $F$ satisfies \eqref{F:ode}. Since the integrals
in \eqref{eq:psi1} and \eqref{eq:psi2} converge for $k=1,2$, and the identities
\eqref{Rtdt}--\eqref{Ptdt} and \eqref{Qtdx}--\eqref{Ptdx} remain valid for $x$ in
the complex plane, we can follow the proof in \cite[section 5.2]{Del2}, making proper
modifications due to some change in notation. 
For convenience of the reader we provide the detailed calculations in 
the appendix (section \ref{sec:appendix}) using the notations of this paper.
\end{proof}

\subsection{Proof of Theorem \ref{theo:ODEsolutions}} \label{proof:ODEsolutions}
We can now prove Theorem \ref{theo:ODEsolutions}
\begin{proof} 
In view of Lemmas \ref{lem:mandpsi} and \ref{lem:psisolutions} we find three solutions of \eqref{m:ODE}. 
After multiplication by appropriate constants these are the solutions $m^{(0)}$, $m^{(2)}$
and $m^{(4)}$ given by the formulas \eqref{m(0)}, \eqref{m(2)}, \eqref{m(4)} in
 Theorem \ref{theo:ODEsolutions}.

The other solutions follow from a symmetry in the system \eqref{m:ODE}.
Namely, if $m(z)$ is  a solution of \eqref{m:ODE} then
\[ \begin{pmatrix} J & O \\ O & - J \end{pmatrix} m(-z), \qquad J = \begin{pmatrix} 0 & 1 \\ 1 & 0 \end{pmatrix} \]
solves \eqref{m:ODE} as well, but with the change of parameters
\begin{equation} \label{parameterchange} r_1 \leftrightarrow r_2, \quad   s_1 \leftrightarrow s_2, \quad \tau \mapsto \tau.
\end{equation}
The constants $C$, $t$ and $\mu$ do not change under this change of parameters,  but
\[ \gamma \mapsto \gamma^{-1}, \qquad \lambda \mapsto -\lambda, \]
see the formulas \eqref{eq:C}--\eqref{eq:t} and \eqref{eq:lambdamu}.
Thus one solution of \eqref{m:ODE} leads to another by a change of sign $z \mapsto -z$,
a change of parameters \eqref{parameterchange}, combined with an interchange $m_1 \leftrightarrow m_2$. 
In this way the solutions $m^{(0)}$, $m^{(2)}$ and $m^{(4)}$ lead to the
solutions $m^{(3)}$, $m^{(5)}$, and $m^{(1)}$. 

This completes the proof of Theorem \ref{theo:ODEsolutions}. 
\end{proof}

\subsection{Proof of Theorem \ref{theo:RHsolution}}  \label{proof:RHsolution}

\begin{proof}
The column vectors of $M$ are in each sector $\Omega_k$ a basis of the vector space
of solutions of \eqref{m:ODE}.   In each sector the  four different  columns represent the 
four different types of asymptotic behavior, as given by \eqref{eq:asymptoticM}. 
Namely, if $e_k$ denotes the $k$th unit vector, 
\begin{equation} \label{eq:Mcolumns}
\begin{aligned}
 M(z) e_1 & \sim e^{-\theta_1(z) + \tau z} \\
	M(z) e_2 & \sim e^{-\theta_2(z) - \tau z} \\
 M(z) e_3 & \sim e^{\theta_1(z) + \tau z} \\
	M(z) e_4 & \sim e^{\theta_2(z) - \tau z} 
	\end{aligned}
	\end{equation}
as $z \to \infty$.

Distinguished solutions of \eqref{m:ODE} are those solutions
that are recessive in a certain sector, i.e., they are smallest possible as $z \to \infty$
in that sector. Recessive solutions are unique up to a multiplicative constant.
From \eqref{eq:Mcolumns} and \eqref{eq:thetas} one sees that
$M(z) e_1$ is a recessive solution of \eqref{m:ODE} in
the sector $S_3$, $M(z) e_2$ is a recessive solution in $S_0$,
$M(z) e_3$ is a recessive solution in  sectors $S_1$ and $S_5$
and $M(z) e_4$ is recessive in sectors $S_2$ and $S_4$.

It will turn out that $m^{(j)}$ is the recessive solution of \eqref{m:ODE} 
in sector $S_j$. This then implies, for example, that $m^{(0)}$ is a multiple of
$M(z) e_2$ in $S_0$. The constant $\sqrt{2\pi} r_2^{1/6}$ in the definition 
\eqref{m(0)} of $m^{(0)}$ has been chosen so that $M(z) e_2 = m^{(0)}$ in $S_0$. 
Since $S_0$ has a non-empty intersection with both $\Omega_0$ and $\Omega_5$ it
then follows that $m^{(0)}$ appears as the second column of $M$ in sectors
$\Omega_0$ and $\Omega_5$. 

Similarly, $m^{(1)}$ is in the third column of $M$ in the sectors $\Omega_0$ and $\Omega_1$,
$m^{(2)}$ is in the fourth column of $M$ in $\Omega_1$ and $\Omega_2$,
and so on.

This leads to the partial solution of the RH problem given in Figure~\ref{fig:recessiveM}.

\begin{figure}[t]
%\vspace{14mm}
\begin{center}
   \setlength{\unitlength}{1truemm}
   \begin{picture}(100,70)(-5,2)
       \put(40,40){\line(1,0){40}}
       \put(40,40){\line(-1,0){40}}
   %    \put(40,40){\line(3,1){30}}
   %    \put(40,40){\line(3,-1){30}}
   %    \put(40,40){\line(-3,1){30}}
   %    \put(40,40){\line(-3,-1){30}}
       \put(40,40){\line(2,3){15}}
       \put(40,40){\line(2,-3){15}}
       \put(40,40){\line(-2,3){15}}
       \put(40,40){\line(-2,-3){15}}
       \put(40,40){\thicklines\circle*{1}}
       \put(39.3,36){$0$}
       \put(60,40){\thicklines\vector(1,0){.0001}}
       \put(20,40){\thicklines\vector(-1,0){.0001}}
    %   \put(55,45){\thicklines\vector(3,1){.0001}}
    %   \put(55,35){\thicklines\vector(3,-1){.0001}}
    %   \put(25,45){\thicklines\vector(-3,1){.0001}}
    %   \put(25,35){\thicklines\vector(-3,-1){.0001}}
       \put(50,55){\thicklines\vector(2,3){.0001}}
       \put(50,25){\thicklines\vector(2,-3){.0001}}
       \put(30,55){\thicklines\vector(-2,3){.0001}}
       \put(30,25){\thicklines\vector(-2,-3){.0001}}

       \put(54,47){$\small{\begin{pmatrix} * & m^{(0)} & m^{(1)} & * \end{pmatrix}}$}
       \put(24,64){$\small{\begin{pmatrix} * & * & m^{(1)} & m^{(2)} \end{pmatrix}}$}
			
       \put(-8,47){$\small{\begin{pmatrix} m^{(3)} & * & * & m^{(2)} \end{pmatrix}}$}
       \put(-7,29){$\small{\begin{pmatrix} m^{(3)} & *  & * & m^{(4)} \end{pmatrix}}$}
       \put(24,14){$\small{\begin{pmatrix} * & * & m^{(5)} & m^{(4)} \end{pmatrix}}$}
       \put(53,29){$\small{\begin{pmatrix} * & m^{(0)} & m^{(5)} & * \end{pmatrix}}$}			  
  \end{picture}
	\caption{Partial solution of the tacnode RH problem with only those columns that are
	recessive solutions of \eqref{m:ODE} in parts of certain sectors. The $*$ entries denote
		columns that still have to be determined. \label{fig:recessiveM}}
\end{center}
\end{figure}
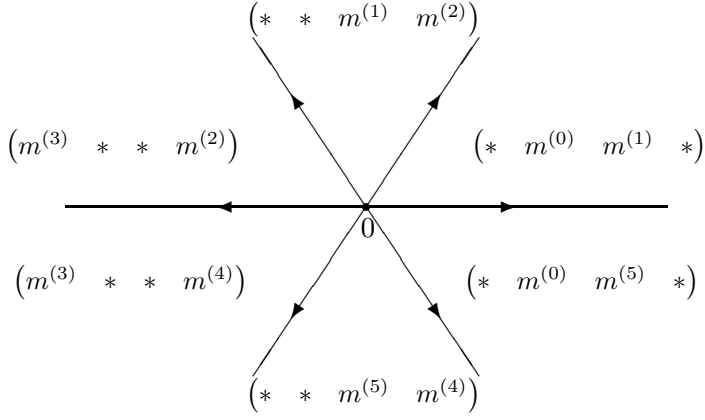

Then we complete the solution by following the effect of the jump matrices
$J_k$, see Figure \ref{fig:tacnodejumps}, and this leads to the full solution
of the tacnode RH problem as given in Figure \ref{fig:solutionM}.

\medskip

We prove in more detail that $m^{(0)}$ is the recessive solution of \eqref{m:ODE} in the sector $S^{(0)}$
and that $m^{(0)} = Me_2$ in $\Omega_0 \cup \Omega_5$. The arguments for the other
solutions $m^{(j)}$ can be done in a similar way.

The asymptotic condition \eqref{eq:asymptoticM} tells us that 
\begin{equation}  \label{M22asympt} 
	M_{22} (z) = \left( \frac{1}{\sqrt{2}} z^{-1/4} + O(z^{-3/4}) \right) e^{- \theta_2(z) - \tau z} 
	\end{equation}
as $z \to \infty$. It will be enough to show that $m_2^{(0)}$ has the same asymptotic behavior
as $z \to \infty$ in the sector $S_0$ in order to conclude that $m^{(0)} = Me_2$ in $\Omega_0 \cup \Omega_5$.

Recall that $m_2^{(0)}$ is given in \eqref{m(0)}.
Since $\lambda - r_2^2 \mu = - \tau$ we can write
\begin{multline} \label{m2(0)} 
	m_2^{(0)}(z) = \sqrt{2\pi} r_2^{1/6} e^{-\tau z}  \Ai\left(r_2^{2/3} z+ \tfrac{2s_2}{r_2^{1/3}} \right)  \\
	 +  \sqrt{2\pi} r_2^{1/6} e^{-\tau z} \int_t^{\infty} 
	\Ai\left(r_2^{2/3} (z + C(x-t)) + \tfrac{2s_2}{r_2^{1/3}}\right) e^{-r_2^2 C \mu(x-t)} R_t(x,t) dx. 
	\end{multline}
From the asymptotic behavior of the Airy function, see \eqref{Airyasympt}, 
we have as $z \to \infty$ with $|\arg z| < \pi - \varepsilon$,
\begin{align} \Ai\left(r_2^{2/3} z + \tfrac{2s_2}{r_2^{1/3}}\right) 
		 = \frac{1}{2 \sqrt{\pi} r_2^{1/6} z^{1/4}}
		e^{- \theta_2(z)}  \left( 1  + O(z^{-1/2}) \right)   \label{Airyrescaledasymptotics}
		\end{align}
(see \eqref{eq:thetas} for $\theta_2$), and
\begin{align*} \Ai\left(r_2^{2/3} (z+C(x-t)) + \tfrac{2s_2}{r_2^{1/3}}\right) 
		= \frac{1}{2 \sqrt{\pi} r_2^{1/6} z^{1/4}}
		e^{- \theta_2(z) - r_2 C(x-t) z^{1/2}} \left( 1 + O(z^{-1/2}) \right)
		\end{align*}
with a $O(z^{-1/2})$ term that is uniform for $x \geq t$ in case $|\arg z| < \pi/6$.
Thus the first term in the right-hand side of \eqref{m2(0)} is 
\begin{equation} \label{m2(0)firstterm} 
	\left( \frac{1}{\sqrt{2}} z^{-1/4} + O(z^{-3/4}) \right)   e^{-\theta_2(z) - \tau z} 
	\end{equation}
as $z \to \infty$, while the second term is (for $z \to \infty$ in $S_0$)
\begin{equation} \label{m2(0)secondterm1}
 O(z^{-1/4}) e^{-\theta_2(z) - \tau z}  \int_t^{\infty} e^{-r_2 C(x-t) z^{1/2}}
	e^{-r_2^2 C \mu(x-t)} R_t(x,t) dx.
	\end{equation}
As $z \to \infty$ in $S_0$, the main contribution to the integral in 
\eqref{m2(0)secondterm1} comes from the endpoint $x = t$. A crude form of Laplace's method,
see e.g.\ \cite{Mil}, shows that the integral is $O(z^{-1/2})$ and therefore we find that the second term 
in the right-hand side of \eqref{m2(0)} is $O(z^{-3/4}) e^{-\theta_2(z) - \tau z}$. Together
with \eqref{m2(0)firstterm} this gives
\begin{equation} \label{m2(0)asympt}
m_2^{(0)}(z) = \left( \frac{1}{\sqrt{2} } z^{-1/4} + O(z^{-3/4}) \right) e^{-\theta_2(z) - \tau z}  	
\end{equation}
as $z \to \infty$ in $S_0$,
which indeed agrees with \eqref{M22asympt}.

This completes the proof of Theorem  \ref{theo:RHsolution}.
\end{proof}

\begin{remark}
A more refined application of Laplace's method would lead to 
\begin{align*}
m_2^{(0)}(z) = \left( z^{-1/4} + \frac{u(t)}{r_2 C}  z^{-3/4} + \cdots \right) \frac{e^{-\theta_2(z) - \tau z}}{\sqrt{2}}
\end{align*}
and similarly
\begin{align*}
m_1^{(0)}(z) 
	= 		\left(- \frac{q(t)}{\gamma \sqrt{r_1r_2} C}  z^{-3/4} + 
		\frac{(r_2^2 \mu + s_2^2) C q(t) - p(t)}{\gamma \sqrt{r_1 r_2} r_2 C^2} z^{-5/4} 
				+ O(z^{-7/4}) \right)  \frac{e^{-\theta_2(z) - \tau z}}{\sqrt{2}}
			\end{align*}
where it is used that $Q_t(t) = q(t)$, $Q_t'(t)  = p(t)$ and $R_t(t,t) = u(t)$.
This leads to formulas for certain entries in $M^{(1)}$ (this is the residue matrix from
the asymptotic condition, see \eqref{eq:asymptoticM}), but we do not discuss that here.
\end{remark}

\section{Correlation kernels} \label{sec:kernels}
As already mentioned in the introduction, there are two correlation kernels
in random matrix theory that can be expressed in terms of the tacnode RH problem,
namely the one-time correlation kernel for the tacnode process \cite{DKZ},
and a critical kernel in the two-matrix model \cite{DG}. The implications of
the explicit form of the solution of the tacnode RH problem for
the tacnode process was discussed by Delvaux \cite{Del2}, who made the connection between
\cite{DKZ} and the different sets of formulas derived in \cite{AFvM,AJvM,FV,Joh}. 
We briefly discuss it in the section \ref{tacnodekernel}. 

The critical kernel in the two-matrix model is due to Duits and Geudens \cite{DG}.
Theorem \ref{theo:RHsolution} yields an explicit integral
representation for this correlation kernel, as we show in section \ref{DGkernel}.

\subsection{Tacnode kernel} \label{tacnodekernel}

The tacnode kernel  is
\begin{equation} \label{Ktac1} 
	K^{tac}(x,y; r_1, r_2, s_1, s_2, \tau) = \frac{1}{2\pi i(x-y)} \begin{pmatrix} 0 & 0 & 1 & 1 \end{pmatrix}
	\widehat{M}^{-1}(y) \widehat{M}(x) \begin{pmatrix} 1 \\ 1 \\ 0 \\ 0 \end{pmatrix} 
	\end{equation}
where $\widehat{M}$ denotes the analytic continuation of the restriction of $M$ to the region
$\Omega_1$ around the positive imaginary axis. This was obtained for  $\tau = 0$
in \cite[Definition 2.6]{DKZ}\footnote{There is a misprint in formula (2.47) in \cite{DKZ}.
It should be $M^{-1}(v) M(u)$ instead of $M^{-1}(u) M(v)$.}  and for general $\tau \in \mathbb R$ in \cite[section 2.2]{Del2}.
Thus by Theorem \ref{theo:RHsolution} we have
\begin{equation} \label{widehatM} 
	\widehat{M} = \begin{pmatrix} m^{(3)} & m^{(0)} & m^{(1)} & m^{(2)} \end{pmatrix},
	\end{equation}
see also Figure \ref{fig:solutionM}.

Because of the symmetry, see \cite[Lemma 5.1]{Del1} or \cite[Lemma 3.1]{Del2},
\begin{equation} \label{Minverse} 
	M^{-1}(z; \tau) = \begin{pmatrix} O & -I \\ I & O \end{pmatrix} M(z; - \tau)^T \begin{pmatrix} O & I \\ -I & O \end{pmatrix} 
	\end{equation}
(we use $M(z; \tau)$ to denote the dependence on $\tau$, and $I$ is the $2 \times 2$ identity matrix),
we can also write
\begin{multline} \label{Ktac2}
 K^{tac}(x,y; r_1, r_2, s_1, s_2, \tau) \\ = \frac{1}{2\pi i(x-y)} \begin{pmatrix} 1 & 1 & 0 & 0 \end{pmatrix}
	 \widehat{M}(y; - \tau)^T \begin{pmatrix} O & I \\ -I & O \end{pmatrix}
	 \widehat{M}(x; \tau) \begin{pmatrix} 1 \\ 1 \\ 0 \\ 0 \end{pmatrix}. 
	\end{multline} 
Thus $K^{tac}$ only depends on the sum of the first two columns of $\widehat{M}$.
If we put 
\begin{equation} \label{eq:hatm}
	\widehat{m} = m^{(0)} + m^{(3)}
	\end{equation} 
	then we get by \eqref{widehatM} and \eqref{Ktac2}
\begin{multline} \label{Ktac3}
	K^{tac}(x,y; r_1, r_2, s_2, s_2, \tau) = \frac{1}{2\pi i (x-y)} 
		 \widehat{m}(y; -\tau)^T \begin{pmatrix} O & I \\ -I & O \end{pmatrix} \widehat{m}(x; \tau)  \\
	= \frac{1}{2\pi i(x-y)}
		\left(-\widehat{m}_1 (x; \tau) \widehat{m}_3 (y, -\tau) + \widehat{m}_3(x; \tau) \widehat{m}_1 (y, -\tau) \right.
		 \\
		 \quad \left.
		 - \widehat{m}_2 (x; \tau) \widehat{m}_4 (y, -\tau)  + \widehat{m}_4 (x; \tau) \widehat{m}_2 (y, -\tau) 	\right).
		\end{multline}

A remarkably simple expression for $\left(\frac{\partial}{\partial x} + \frac{\partial}{\partial y} \right)
	K^{tac}(x,y)$ can be obtained from \eqref{Ktac1}. Using the differential equation \eqref{M:ODE} for $\widehat{M}$ 
	and the formula \eqref{eq:U} for $U$, we obtain
\begin{multline*} 
\left(\frac{\partial}{\partial x} + \frac{\partial}{\partial y} \right) 
	K^{tac}(x,y)  = \\
		\frac{1}{2\pi i(x-y)} \begin{pmatrix} 0 & 0 & 1 & 1 \end{pmatrix} \widehat{M}^{-1}(y) \left(U(x) - U(y)\right) \widehat{M}(x) 
		\begin{pmatrix} 1 \\ 1\\ 0 \\ 0 \end{pmatrix} \\
			 = 
			\frac{1}{2\pi} \begin{pmatrix} 0 & 0 & 1 & 1 \end{pmatrix} \widehat{M}^{-1}(y) \left( r_1 E_{3,1} - r_2 E_{4,2} \right)
				 \widehat{M}(x) 
		\begin{pmatrix} 1 \\ 1\\ 0 \\ 0 \end{pmatrix} 
		\end{multline*}
where $E_{j,k}$ is the matrix with $1$ in position $j,k$ and $0$ otherwise.
Combining this with \eqref{Minverse} we get
\begin{align}  \label{Ktac4}
\left(\frac{\partial}{\partial x} + \frac{\partial}{\partial y} \right) 
	K^{tac}(x,y)
%	= 
%	\frac{1}{2\pi} \begin{pmatrix} 0 & 0 & 1 & 1 \end{pmatrix} \widehat{M}(y; - \tau)^T \diag \begin{pmatrix} r_1 & r_2 & 0 & 0 \end{pmatrix}
%	 \widehat{M}(x; \tau) \begin{pmatrix} 1 \\ 1\\ 0 \\ 0 \end{pmatrix}  \\ 
	 = 	\frac{1}{2\pi} \left(r_1 \widehat{m}_1(y;-\tau) \widehat{m}_1(x; \tau)
	-r_2 \widehat{m}_2(y;-\tau) \widehat{m}_2(x; \tau) \right).
	% \sum_{j=1}^2 (-1)^{j-1} r_j \widehat{m}_j(y;-\tau) \widehat{m}_j(x; \tau).
		\end{align}
where $\widehat{m} = m^{(0)} + m^{(3)}$ as in \eqref{eq:hatm}.	

Delvaux \cite{Del2} further analyzed \eqref{Ktac3}, \eqref{Ktac4} using the formulas \eqref{m(0)} and \eqref{m(3)}
for $m^{(0)}$ and $m^{(3)}$, and showed that the expression
for the tacnode kernel agrees with the one given by Ferrari and Vet\H{o} \cite{FV}.
		
\subsection{Duits-Geudens critical kernel} \label{DGkernel}

The Duits-Geudens  kernel appears in a critical regime in the two-matrix model \cite{DK},
where it was obtained from a Riemann-Hilbert analysis based on \cite{DK,DKM}. It
is a remarkable fact that it can be expressed in terms of the solution of the tacnode RH problem
\eqref{eq:jumpsM}--\eqref{eq:asymptoticM} 
with the special choice of parameters 
\begin{equation} \label{DGparameters} 
	r_1 = r_2 = 1, \qquad s_1 = s_2  = s, \qquad \tau \in \mathbb R.
	\end{equation}
With those parameters we have by \eqref{eq:C}--\eqref{eq:t}, and \eqref{eq:lambdamu},
\[ C = 2^{1/3}, \quad \gamma = 1, \quad t = 2^{2/3} (2s - \tau^2),  \quad \lambda = 0, \quad \mu = \tau. \]
The formula for the critical kernel is
\begin{equation} \label{DGkernel1} 
	K^{crit}(x,y;s,\tau) = 
	\frac{1}{2\pi i (x-y)} \begin{pmatrix} -1 & 1 & 0 & 0 \end{pmatrix}
		M(ix)^{-1} M(iy) \begin{pmatrix} 1 \\ 1 \\ 0 \\ 0 \end{pmatrix},
	\end{equation}
which can easily be obtained from the formulas (2.13) and (2.15) in \cite{DG}.
Here $M$ is the solution of the tacnode RH problem  with parameters \eqref{DGparameters}.

We can use \eqref{Minverse} and Theorem \ref{theo:RHsolution} to rewrite \eqref{DGkernel1} as
\begin{multline} \label{DGkernel2} 
	K^{crit}(x,y; s, \tau) 
	\\ =
	\frac{1}{2\pi i (x-y)} \begin{pmatrix} 0 & 0 & 1 & -1 \end{pmatrix}
		M(ix;-\tau)^T \begin{pmatrix} O & I \\ -I & O \end{pmatrix}  M(iy; \tau) \begin{pmatrix} 1 \\ 1 \\ 0 \\ 0 \end{pmatrix} \\
		= \frac{1}{2\pi i (x-y)}
			(m^{(1)} - m^{(2)})(ix, -\tau)^T \begin{pmatrix} O & I \\ -I & O \end{pmatrix} 
			(m^{(0)} + m^{(3)})(iy, \tau).
			\end{multline}
The critical kernel thus depends on $\widehat{m} = m^{(0)} + m^{(3)}$, see \eqref{eq:hatm}, 
which also appeared in the tacnode kernel, and on
\begin{equation} \label{eq:tildem}
	\widetilde{m} = m^{(1)} - m^{(2)} = m^{(5)} - m^{(4)}.
	\end{equation}
The second identity in \eqref{eq:tildem} holds because of \eqref{m:consistent}.
	
Another formula for $K^{crit}$ comes from differentiating \eqref{DGkernel1} 
with respect to $s$.	There are differential equations
$\frac{\partial M}{\partial s_1} = V_1 M$ and $\frac{\partial M}{\partial s_2} = V_2 M$,
which if $s = s_1 =s_2$ as in \eqref{DGparameters} leads to
\[ \frac{\partial M}{\partial s} = VM, \qquad V = V_1 + V_2, \]
and $V$ is given explicitly in \cite[Proposition 5.11]{Del1}. This formula implies
\begin{equation} \label{VxVy} 
	V(x) - V(y) = -2i(x-y) \left( E_{3,1} + E_{4,2} \right),
	\end{equation}
from which it follows that
\begin{align} \nonumber
	\frac{\partial}{\partial s} M^{-1}(ix) M(iy) & =  M^{-1}(ix) \left(V(iy) - V(ix) \right) M(iy) \\
		& = -2 (x-y) M^{-1}(ix) \left(E_{3,1} + E_{4,2} \right) M(iy)
\end{align}
Thus from \eqref{DGkernel1} and \eqref{Minverse} we get
\begin{multline} \label{DGkernel3}
	\frac{\partial}{\partial s} K^{crit}(x,y;s,\tau)  \\
	= 
	\frac{-1}{\pi i} \begin{pmatrix} -1 & 1 & 0 & 0 \end{pmatrix}
		M(ix)^{-1} (E_{3,1} + E_{4,2}) M(iy) \begin{pmatrix} 1 \\ 1 \\ 0 \\ 0 \end{pmatrix} \\
		= 	\frac{-1}{\pi i} \begin{pmatrix} 0 & 0 & 1 & -1  \end{pmatrix}
		M(ix; -\tau)^{T} (E_{1,1} + E_{2,2}) M(iy; \tau) 
		\begin{pmatrix} 1 \\ 1 \\ 0 \\ 0 \end{pmatrix},
	\end{multline}
which is somewhat similar to the expression \eqref{Ktac4} for the tacnode kernel,
except that the two solutions \eqref{eq:hatm} and \eqref{eq:tildem} of the 
ODE \eqref{m:ODE} are now involved.
Indeed, by \eqref{DGkernel3} and the solution of the tacnode RH problem
\begin{multline}  \label{DGkernel4}
	\frac{\partial}{\partial s} K^{crit}(x,y;s,\tau)  
	= 	
	\frac{-1}{\pi i} \widetilde{m}(ix;-\tau)^T (E_{11} + E_{22}) \widehat{m}(iy; \tau) \\
	= \frac{-1}{\pi i} \left[ \widetilde{m}_1(ix; -\tau)  \widehat{m}_1(iy; \tau) 
	+ \tilde{m}_2(ix; -\tau)  \widehat{m}_2(iy; \tau) \right],
	\end{multline}
	which is a rank two kernel.

Let us check that \eqref{DGkernel4} is real-valued.
The symmetries of the tacnode RH problem, see \cite[Lemma 5.1]{Del1},
\begin{align*} 
	\overline{M(z)} & = \begin{pmatrix} I & O \\ O & - I \end{pmatrix}
	M(\overline{z}) \begin{pmatrix} I & O \\ O & - I \end{pmatrix}, \\
	M(-z) & = \begin{pmatrix} J & O \\ O & - J \end{pmatrix}
	M(z) \begin{pmatrix} J & O \\ O & - J \end{pmatrix}, 
		\qquad J = \begin{pmatrix} 0 & 1 \\ 1 & 0 \end{pmatrix}. 
		\end{align*}
imply for $z = ix$ with $x$ real,
\[ \overline{M(ix)} = \begin{pmatrix} J & O \\ O & J \end{pmatrix}
	M(ix) \begin{pmatrix} J & O \\ O & J \end{pmatrix}. \]
In view of the solution for $M$ in Theorem \ref{theo:RHsolution},
this means that for real $x, y$, 
\[ \overline{m_{1}^{(1)}(ix)} = m_2^{(2)}(ix), \quad \overline{m_1^{(2)}(ix)} = m_2^{(2)}(ix), \]
\[ \overline{m_{1}^{(0)}(iy)} = m_2^{(3)}(iy), \quad \overline{m_1^{(3)}(iy)} = m_2^{(0)}(iy). \]
Using this and \eqref{eq:hatm}, \eqref{eq:tildem} in \eqref{DGkernel4}, we find
\begin{equation} \label{DGkernel5}
	\frac{\partial}{\partial s} K^{crit}(x,y;s,\tau)  
	= 	\frac{-2}{\pi} \Im \left[ \tilde{m}_1(ix,s, -\tau)  \widehat{m}_1(iy, s, \tau) \right] 
	\end{equation}
which is real-valued (as it should be).

Since $K^{crit}(x,y; s, \tau) \to 0$ as $s \to +\infty$, we
recover $K^{crit}$ from \eqref{DGkernel5} after integration with respect to $s$
\begin{equation} \label{DGkernel6} 
	K^{crit}(x,y;s,\tau)  
	= 	\frac{2}{\pi} \int_s^{\infty} 
		\Im \left[ \tilde{m}_1 (ix,s', -\tau)  \widehat{m}_1(iy, s', \tau) \right] ds'
	\end{equation}
which is maybe the simplest form for the Duits-Geudens critical kernel.
	
\section{Appendix: Proof of Lemma \ref{lem:psisolutions}} \label{sec:appendix}

\begin{proof}
Throughout the proof of Lemma \ref{lem:psisolutions} we simply 
write $\infty$ instead of $\infty \omega^{2k}$, but it is of
course understood that the integrals extend to infinity in the appropriate direction.

In addition to the differential identities \eqref{Rtdt}--\eqref{Ptdx}
for the functions $Q_t$, $P_t$, and $R_t$, there are further identities
in \cite{TW,TW2} that involve the four functions of the variable $t$ defined by
\begin{align} \label{qpuv}
	q(t)  = Q_t(t), \quad
	p(t)  = P_t(t), \quad
	u(t)  = R_t(t,t), \quad
	v(t)  = \frac{1}{2}(u^2 - q^2).
\end{align}
These four functions satisfy the closed differential system
\begin{align} \label{qpuvdt}
	q' = p - uq, \quad
	p' = tq + uq - 2v q, \quad
	u' = - q^2, \quad
	v' = - pq.
\end{align}

The second order system \eqref{psi:ODE} gives formulas for $\frac{\partial^2}{\partial z^2} \psi_1$
and $\frac{\partial^2}{\partial z^2} \psi_2$. We start by verifying the latter.

\medskip

From \eqref{eq:psi2} and \eqref{F:ode} we obtain
\begin{align*} 
	\frac{\partial^2}{\partial z^2} \psi_2 & = -2 r_2^2 \mu \frac{\partial}{\partial z} \psi_2 
	+ (r_2^2 z + 2 r_2 s_2 - r_2^4 \mu^2) \psi_2  \\
	& \quad + C r_2^2 \int_t^{\infty} (x-t) F(z+C(x-t)) R_t(x,t) dx.
	\end{align*}
Here we use \eqref{Rt}, \eqref{Qtdx} and \eqref{qpuvdt} to obtain
\[ (x-t) R_t(x,t) = pQ_t(x) - q P_t(x) = q^2 R_t(x,t) + q' Q_t(x) - q  Q_t'(x), \]
where $p = p(t)$. Thus by \eqref{eq:psi1} and \eqref{eq:psi2},
\begin{align*} 
	\frac{\partial^2}{\partial z^2} \psi_2 & = -2 r_2^2 \mu \frac{\partial}{\partial z} \psi_2 +  (r_2^2 z + 2 r_2 s_2 - r_2^4 \mu^2) \psi_2  \\
	& + C r_2^2 \int_t^{\infty} F(z+C(x-t)) (q^2 R_t(x,t) + q' Q_t(x) - q Q_t'(x)) dx \\
	& = -2 r_2^2 \mu \frac{\partial}{\partial z} \psi_2 + (r_2^2 z + 2 r_2 s_2 - r_2^4 \mu^2) \psi_2 + C r_2^2 q^2 (\psi_2 - F(z)) \\
	&   - C r_2^2 q'\psi_1  - C r_2^2 q \int_t^{\infty} F(z+C(x-t)) Q_t'(x) dx.
	\end{align*}
We apply integration by parts to the remaining integral. The integrated term is $C r_2^2 q^2 F(z)$  
and so we obtain (again using \eqref{eq:psi1})
\begin{align*} 
	\frac{\partial^2}{\partial z^2} \psi_2  & = -2 r_2^2 \mu \frac{\partial}{\partial z} \psi_2 + (r_2^2 z + 2 r_2 s_2 + Cr_2^2 q^2- r_2^4 \mu^2) \psi_2  \\
	&   - Cr_2^2 q'\psi_1 +  C^2 r_2^2 q \int_t^{\infty} F'(z+C(x-t)) Q_t(x) dx, \\
	& = -2 r_2^2 \mu \frac{\partial}{\partial z} \psi_2 + (r_2^2 z + 2 r_2 s_2 + Cr_2^2 q^2- r_2^4 \mu^2) \psi_2  \\
	&   - C^2 r_2^2 q \frac{\partial}{\partial z} \psi_1 - Cr_2^2 q'\psi_1
	\end{align*}
which is what is needed for  $\frac{\partial^2}{\partial z^2} \psi_2$ according to \eqref{psi:ODE}.

\medskip

The proof for $\frac{\partial^2}{\partial z^2} \psi_1$ is a bit more involved. We have by  \eqref{eq:psi1} and \eqref{F:ode} 
\begin{align} \nonumber
\frac{\partial^2}{\partial z^2} \psi_1 & =  - \int_t^{\infty} F''(z + C(x-t)) Q_t(x) dx \\
& \nonumber = - 2 r_2^2 \mu \frac{\partial}{\partial z} \psi_1 + (r_2^2 z + 2r_2 s_2 - r_2^4 \mu^2 - Cr_2^2 t)  \psi_1 \\
& \quad    - Cr_2^2 \int_t^{\infty} x F(z+ C(x-t)) Q_t(x) dx.
\label{xQtintegral} 
	\end{align}
Here we use \eqref{Ptdx} and \eqref{Qtdx} to obtain
\begin{align*} x Q_t(x) & = 2v Q_t(x) - p R_t(x,t) + P_t'(x) - uP_t(x) \\
	& = - (u^2 - 2v) Q_t(x) - (p-uq) R_t(x,t) + P_t'(x) - u Q_t'(x),
	\end{align*}
	which by \eqref{qpuvdt} and the formula for $v(t)$ in \eqref{qpuv} gives
\[ x Q_t(x) = - q^2 Q_t(x) - q'R_t(x,t) + P_t'(x) - u Q_t'(x). \]
Using this in \eqref{xQtintegral} we find using the definitions \eqref{eq:psi1} and \eqref{eq:psi2},
\begin{align*}
\frac{\partial^2}{\partial z^2} \psi_1 & =  
		- 2 r_2^2 \mu \frac{\partial}{\partial z} \psi_1 + (r_2^2 z + 2r_2 s_2 - r_2^4 \mu^2 - Cr_2^2 t - C r_2^2 q^2)  \psi_1 \\
	& + C r_2^2 q' (\psi_2 - F(z)) - Cr_2^2 \int_t^{\infty} F(z + C(x-t)) (P_t'(x) - uQ_t'(x)) dx. 
	\end{align*}
We integrate by parts on the remaining integral. The term $-Cr_2^2 q' F(z)$ is cancelled by the
integrated term $C r_2^2 (p -u q) F(z)$ because of \eqref{qpuvdt}. Thus
\begin{align*}
\frac{\partial^2}{\partial z^2} \psi_1 & = 
		-2 r_2^2 \mu \frac{\partial}{\partial z} \psi_1 + (r_2^2 z + 2r_2 s_2 - r_2^4 \mu^2 - Cr_2^2 t - C r_2^2 q^2)  \psi_1 \\
	&    + Cr_2^2 q' \psi_2 + C^2 r_2^2 \int_t^{\infty} F'(z+C(x-t)) (P_t(x) - uQ_t(x)) dx. 
	\end{align*}
Now we use \eqref{Qtdx} again and we simplify using \eqref{eq:psi2}
\begin{align*}
\frac{\partial^2}{\partial z^2} \psi_1 & = 
		- 2 r_2^2 \mu \frac{\partial}{\partial z} \psi_1 + (r_2^2 z + 2r_2 s_2 -r_2^4 \mu^2 - Cr_2^2 t - C r_2^2 q^2)  \psi_1 \\
	& + C r_2^2 q' \psi_2  - C^2 r_2^2 \int_t^{\infty} F'(z+ C(x-t)) (qR_t(x,t) - Q_t'(x)) dx \\
	& = 
		- 2 r_2^2 \mu \frac{\partial}{\partial z} \psi_1 + (r_2^2 z + 2r_2 s_2 - r_2^4 \mu^2 - Cr_2^2 t - C r_2^2 q^2)  \psi_1 \\
	& + C r_2^2 q' \psi_2 - C^2 r_2^2 q \left( \frac{\partial}{\partial z} \psi_2 - F'(z) \right) + C^2 r_2^2  \int_t^{\infty} F'(z+ C(x-t)) Q_t'(x) dx.
	\end{align*}
We integrate by parts, the integrated term cancels with $C^2 r_2^2 q F'(z)$ and so
\begin{align*}
\frac{\partial^2}{\partial z^2} \psi_1 & = 
		- 2 r_2^2 \mu \frac{\partial}{\partial z} \psi_1 + (r_2^2 z + 2r_2 s_2 - r_2^4 \mu^2 - Cr_2^2 t - C r_2^2 q^2)  \psi_1 \\
	& - C^2 r_2^2 q  \frac{\partial}{\partial z} \psi_2 + C r_2^2 q' \psi_2  - C^3 r_2^2  \int_t^{\infty} F''(z+ C(x-t)) Q_t(x) dx.
	\end{align*}
The final integral we can express in terms of $\frac{\partial^2}{\partial z^2} \psi_1$ (see the
first identity in \eqref{xQtintegral}) and the result is that
\begin{align*}
\frac{\partial^2}{\partial z^2} \psi_1 & = 
		-2 r_2^2 \mu \frac{\partial}{\partial z} \psi_1 + (r_2^2 z + 2r_2 s_2 - r_2^4 \mu^2 - Cr_2^2 t - C r_2^2 q^2)  \psi_1 \\
	& - C^2 r_2^2 q  \frac{\partial}{\partial z} \psi_2 + C r_2^2 q' \psi_2  + C^3 r_2^2  \frac{\partial^2}{\partial z^2} \psi_1.
	\end{align*}
Rearranging terms we find
\begin{align*}
(C^3 r_2^2 - 1)  \frac{\partial^2}{\partial z^2} \psi_1 & =
	+ C^2 r_2^2 q \frac{\partial}{\partial z} \psi_2 + 2 r_2^2 \mu \frac{\partial}{\partial z} \psi_1 \\
	& - C r_2^2 q' \psi_2 + (- r_2^2 z - 2r_2s_2 + r_2^4 \mu^2 + C r_2^2 t + C r_2^2 q^2) \psi_1.
	\end{align*}

Since $C^3 r_2^2 - 1 = (r_2/r_1)^2$, this is
\begin{align*}	
\frac{\partial^2}{\partial z^2} \psi_1 & =
	C^2 r_1^2 q \frac{\partial}{\partial z} \psi_2 + 2 r_1^2 \mu \frac{\partial}{\partial z} \psi_1 \\
	& - C r_1^2 q' \psi_2 + (-r_1^2 z - 2 \frac{r_1^2}{r_2} s_2 + r_1^2 r_2^2 \mu^2 + C r_1^2 t + C r_1^2 q^2) \psi_1.
	\end{align*}
Since $ t= C^{-1} \left( 2\frac{s_1}{r_1} + 2\frac{s_2}{r_2} - (r_1^2 + r_2^2) \mu^2 \right)$, we
finally obtain
\begin{align*}	
\frac{\partial^2}{\partial z^2} \psi_1 & =
	C^2 r_1^2 q \frac{\partial}{\partial z} \psi_2 + 2 r_1^2 \mu \frac{\partial}{\partial z} \psi_1 \\
	& - C r_1^2 q' \psi_2 + (-r_1^2 z + 2 r_1 s_1 - r_1^4 \mu^2 + C r_1^2 q^2) \psi_1
	\end{align*}
	as claimed in  \eqref{psi:ODE}.
\end{proof}	

\section*{Acknowledgements}

I am grateful to Steven Delvaux and Dries Geudens for useful discussions
about the tacnode problem and the critical kernels.

The author is supported by KU Leuven Research Grant No. OT/12/073,
the Belgian Interuniversity Attraction Pole P07/18,
FWO Flanders projects G.0641.11 and G.0934.13, and 
by Grant No. MTM2011-28952-C02 of the Spanish Ministry of Science and Innovation.

\end{document}